\documentclass[11pt,a4paper,reqno]{amsart}

\usepackage{amscd,amsmath}
\usepackage{amssymb,amsmath,latexsym,color,enumitem,tikz,dsfont,tikz-cd,hyperref}
\usepackage{amsaddr}
\usepackage[all]{xypic}       
\usepackage[varg]{pxfonts}
\usepackage{tabu,adjustbox}
\usepackage{graphicx}

\makeatletter
\@namedef{subjclassname@2010}{2020 Mathematics Subject Classification}
\makeatother

\makeatletter
\newcommand*{\@old@slash}{}\let\@old@slash\slash
\def\slash{\relax\ifmmode\delimiter"502F30E\mathopen{}\else\@old@slash\fi}
\makeatother

\def\CC{\mathcal{C}}
\newcommand{\cat}[1]{\ensuremath{\mathsf{#1}}} 

 \newcommand{\newdownarrow}{{{\rlap{$\ $}\hbox{$\downarrow$}}}}
 \newcommand{\newuparrow}{{{\rlap{$\ $}\hbox{$\uparrow$}}}}

 \newcommand{\tbigvee}{\mathop{\textstyle \bigvee }}
 \newcommand{\tbigwedge}{\mathop{\textstyle \bigwedge }}
 \newcommand{\tbigsqcup}{\mathop{\textstyle \bigsqcup }}

\newcommand{\donotbreakdash}[1]{#1\nobreakdash-\hspace{0pt}}

\newtheorem{theorem}{Theorem}[section]
\newtheorem{proposition}[theorem]{Proposition}
\newtheorem{lemma}[theorem]{Lemma}
\newtheorem{corollary}[theorem]{Corollary}
\newtheorem{fact}[theorem]{Fact}

\theoremstyle{definition}

\newtheorem{example}[theorem]{Example}

\theoremstyle{remark}
\newtheorem{remark}[theorem]{Remark}
\newtheorem{remarks}[theorem]{Remarks}

\title[On infinite variants of De Morgan law in locale theory]{On infinite variants of De Morgan law\\
in locale theory}

\author[Igor Arrieta]{Igor Arrieta}
\address{\footnotesize{CMUC, Department of Mathematics, University of
Coimbra,\\ 3001-501 Coimbra, Portugal}}
\address{\footnotesize{Department of Mathematics, University of the Basque Country, UPV/EHU,\\ 48080 Bilbao, Spain}}
 \email{iarrieta024@ikasle.ehu.eus}

\begin{document}

\begin{abstract}
A locale, being a complete Heyting algebra, satisfies De Morgan law 
$(a\vee b)^*=a^*\wedge b^*$ for pseudocomplements. The dual De Morgan 
law $(a\wedge b)^*={a^* \vee b^*}$ (here referred to as the second De 
Morgan law) is equivalent to, among other conditions, $(a\vee b)^{**} =a^{**}\vee b^{**}$, and characterizes the class of extremally disconnected 
locales. This paper presents a study of the subclasses of extremally 
disconnected locales determined by the infinite versions of the second 
De Morgan law and its equivalents.\vspace*{-2.5em}
\end{abstract}
\maketitle
\hypersetup{pageanchor=true}
 \section{Introduction}
Recall that a topological space $X$ is \emph{extremally disconnected} if the closure of every open set is open. The point-free counterpart of this notion is that of an \emph{extremally disconnected locale}, that is, of a locale $L$ satisfying
\begin{gather}\label{edframebat}\tag{ED} (a\wedge b)^{*} =a^{*}\vee b^{*}, \qquad \text{for every } a,b\in L, \end{gather}
or, equivalently, 
\begin{equation}\label{edframebi}\tag{ED$^\prime$} (a\vee b)^{**} =a^{**}\vee b^{**}, \qquad \text{for every } a,b\in L. \end{equation}

Extremal disconnectedness is a well-established topic both in classical and point-free topology (see, for example, Section~3.5 and the following ones in Johnstone's monograph \cite{STONE}), and it admits various characterizations (cf. Proposition~\ref{EDCAR} below, or Proposition~2.3 in \cite{PARAL}). 
The goal of the present paper is to study infinite versions of conditions  \eqref{edframebat} and \eqref{edframebi} (cf. Proposition~\ref{EDCAR}\,\ref{EDCAR1} and \ref{EDCAR3}). We show that, when considered in the infinite case, the conditions are no longer equivalent, and they define two different properties strictly between Booleaness (denoted by (CBA) for short) and extremal disconnectedness.

The stronger one corresponds to the infinite second De Morgan law (IDM), and it can be expressed as the conjunction of the weaker one (which we call infinite extremal disconnectedness, (IED) for the sake of brevity) and a weak scatteredness condition. We are thus led to study the chain of strict implications
   \[\text{(CBA)}\implies\text{(\ref{IDM})}\implies \text{(\ref{IED})}\implies\text{(\ref{edframebat})}.\]
The paper is organized as follows.  In Section~\ref{sec2}, we introduce preliminaries explaining basic concepts and introducing notation. We also recall several aspects about De Morgan law in locale theory.  Section~\ref{sec3} is devoted to define and study infinite variants of extremal disconnectedness. In Section~\ref{sec4} the auxiliary scatteredness property is discussed, since it has not been fully investigated in locale theory and looks interesting in its own. In Sections~\ref{sec5} and \ref{sec6} additional characterizations of (hereditary) infinite extremal disconnectedness are provided. Section \ref{sec7} concerns some categorical aspects mainly regarding the (non-) functoriality of Booleanization. By using infinite extremal disconnectedness an alternative framework for making Booleanization functorial is discussed.  Furthermore, a construction parallel to DeMorganization  (\cite{DM}, cf. \cite{FIELDS}) is provided.

\section{Preliminaries}\label{sec2}

We first recall some background on point-free topology. For more information on the categories of frames and locales, we refer the reader to Johnstone \cite{STONE} or the more recent Picado-Pultr \cite{PP12}. A  \emph{locale} (or \emph{frame}) is a complete lattice $L$ satisfying $$a\wedge \tbigvee B = \tbigvee \{ a\wedge b\mid b\in B\}$$
for all $a\in L$ and $B\subseteq L$. A \emph{frame homomorphism} is a function preserving arbitrary joins (including the bottom element $0$) and finite meets (including the top element $1$). Frames and their homomorphisms form a category \cat{Frm}. For each $a$ the map $a\wedge(-)$ preserves arbitrary joins, thus it has a right (Galois) adjoint $a \to (-)$,  making $L$ a complete Heyting algebra (i.e. a cartesian closed category, if one regards $L$ as a  thin category). This right adjoint is called the \emph{Heyting operator}. 
A comprehensive list  
 of its properties may be found in \cite[III 3.1.1]{PP12}, and we will freely use some of them, e.g.:
\begin{enumerate}[label={\rm(\arabic*)}]
\item \label{hp1}  $a\to \tbigwedge b_i = \tbigwedge (a \to b_i)$\textup;
\item \label{hp2} $\left( \tbigvee a_i \right) \to b= \tbigwedge (a_i\to b)$\textup; \hfill\textup($\tbigvee \tbigwedge$-distr\textup)
\item \label{hp3} $a\to b=1$ if and only if $a\leq b$\textup;
\item \label{hp4} $a\to b=a\to (a\wedge b)$\textup;
\item \label{hp5} $a\to b=(a\vee b)\to b$\textup;
\end{enumerate}
 for all $\{a_i\}_{i\in I},\,\{b_i\}_{i\in I}\subseteq L$ and all $a,b\in L$. 
  
The \emph{pseudocomplement} of an $a\in L$ is $a^*=a\to0$. We shall also use standard facts such as $a\leq a^{**}$, $a^{***}=a^*$, or the fact that $a\leq b$ implies $b^*\leq a^*$. As a particular case of \textup($\tbigvee \tbigwedge$-distr\textup) above, one has $\left( \tbigvee a_i\right)^*=\tbigwedge a_{i}^{*}$  for any family $\{a_i\}_{i\in I}\subseteq L$. This is known as the \emph{first De Morgan law} (see, for example, \cite[A.I.~Proposition~7.3.3]{PP12}).

Given a topological space $X$, its lattice of open sets $\Omega(X)$ is a frame, and this construction can be upgraded to a functor $\Omega\colon \cat{Top}\longrightarrow \cat{Frm}^{op}$ which sends a continuous map $f\colon X\to Y$ to the preimage operator \linebreak$f^{-1}[-]\colon \Omega(Y)\to\Omega(X)$. There is a further \emph{spectrum} functor $\Sigma\colon \cat{Frm}^{op}\longrightarrow \cat{Top}$ which yields an adjunction 
\[
\begin{tikzcd}[row sep=huge, column sep=huge, text height=1.5ex, text depth=0.25ex]
 \cat{Top} \arrow[bend left=20]{r}[name=a]{\Omega} &\cat{Frm}^{op}  \arrow[bend left=20]{l}[name=b]{\Sigma} \arrow[phantom, from=a, to=b]{}{\ \bot}
\end{tikzcd}
\]

The category $\cat{Loc}$ of locales is by definition the opposite category of $\cat{Frm}$: $$\cat{Loc}=\cat{Frm}^{op},$$
and $\Omega$ restricts to a full embedding of a substantial part of $\cat{Top}$ (namely the full subcategory of sober spaces) into $\cat{Loc}$. The latter can therefore be seen as a category of generalized spaces. We shall mostly speak of objects in \cat{Loc} as locales (instead of frames) when emphasizing the covariant approach. Morphisms in \cat{Loc} can be concretely represented by the right (Galois) adjoints $f_*\colon M\longrightarrow L$ of the corresponding frame homomorphisms $f\colon L\longrightarrow M$; these will be referred to as \emph{localic maps}.

A regular subobject in $\cat{Loc}$ (that is, an isomorphism class of regular monomorphisms) of a locale $L$ is \emph{a sublocale} of $L$. Sublocales of a locale $L$ can be represented as the actual subsets $S\subseteq L$ such that \begin{enumerate}[label={\rm(\arabic*)}]
\item \label{subp1} $S$ is closed under arbitrary meets in $L$, and 
\item \label{subp2} $a\to s\in S$ for all $a\in L$ and $s\in S$.  \end{enumerate}
A different, but equivalent, representation of sublocales is by means of nuclei, i.e. inflationary and idempotent maps $\nu\colon L\longrightarrow L$ which preserve binary meets. The sublocale associated to a nucleus $\nu$ is the image $\nu[L]$, and conversely the nucleus associated to a sublocale $S\subseteq L$ is given by $\iota_S \circ \nu_S$  where $\iota_S$ denotes the inclusion of $S$ into $L$ and $\nu_S$ is its left adjoint frame homomorphism given by 
$$\nu_S(a)=\tbigwedge \{ s\in S \mid s\geq a \}.$$ The following identity is satisfied
$$a\to s=  \nu_S(a)\to s\quad \textrm{for all }a\in L,\,s\in S.$$
A sublocale should not be confused with a \emph{subframe}; the latter is a subobject of a locale in the category $\cat{Frm}$. Subframes can be represented as subsets which are closed under arbitrary joins and finite meets.

For each $a\in L$, one has an \emph{open sublocale} and a \emph{closed sublocale}
$$\mathfrak{o}(a)=\{b\mid b=a\to b\}=\{a\to b\mid b\in L\} \quad\textrm{and}\quad \mathfrak{c}(a)=\newuparrow a$$
which in the spatial case $L=\Omega(X)$ correspond to the open and closed subspaces. 

If $S$ is a sublocale of $L$, the \emph{closure} of $S$ in $L$, denoted by $\overline{S}$, is the smallest closed sublocale containing $S$, which can be computed as $\overline{S}=\mathfrak{c}(\tbigwedge S)$. 
A sublocale $S$ is \emph{dense} if $\overline{S}=L$, or equivalently if $0\in S$. The family $$\mathcal{S}(L)$$ of all sublocales of $L$ partially ordered by inclusion is always a coframe (the order-theoretic dual of a frame) and every localic map $f\colon L\longrightarrow M$  gives by pulling back in $\cat{Loc}$ an \emph{inverse image} map $f^{-1}[-]\colon \mathcal{S}(M)\longrightarrow \mathcal{S}(L)$ which turns out to be a coframe homomorphism (i.e. a function which preserves arbitrary meets and finite joins). Set-theoretic direct image yields a \emph{direct image} map
$f[-] \colon \mathcal{S}(L)\longrightarrow \mathcal{S}(M)$ which is additionally a colocalic map (i.e. a left adjoint of a coframe homomorphism). In this context, the usual adjunction $f[-] \dashv f^{-1}[-]$ is satisfied.
\medskip

Given a locale $L$, we denote by  $B_L$ the subset consisting of \emph{regular elements} of $L$; that is, those $a\in L$ with $a^{**}=a$, or equivalently those $a\in L$ with $a=b^*$ for some $b\in L$. In other words, one has  $$B_L=\{ a^*\mid a\in L\} =\{a\in L\mid a^{**}=a\};$$
 and this subset is called the \emph{Booleanization} of $L$. It was originally introduced by Glivenko \cite{GLIVENKO} as a generalization of the Boolean algebra of regular open subspaces of a topological space, and it can be characterized in several ways, e.g. it is the least dense sublocale, or equivalently, the unique Boolean dense sublocale.

 The nucleus associated to $B_L$ is the double-negation map $(-)^{**}\colon L\longrightarrow L$. We recall the following facts:

\begin{proposition}\label{basic}
Let $L$ be a locale. Then the following hold\textup:
 \begin{enumerate}[label={\rm(\arabic*)}]
 \item \label{basic1} $\tbigwedge a_{i}^{**}\le (\tbigwedge a_{i})^{**}$ for all $\{a_i\}_{i=1}^n\subseteq L$.
 \item \label{basic2} $(a\to b)^{**}= a^{**}\to b^{**}$ for all $a,b\in L$.
 \end{enumerate}
\end{proposition}

Since the reverse inequality in \ref{basic1} above is trivially satisfied, this tells us that the nucleus  $(-)^{**}\colon L\longrightarrow L$ \emph{always} preserves finite meets and the Heyting operation. Equivalently, the frame homomorphism $(-)^{**}\colon L\to B_L$ is a Heyting algebra homomorphism.

Note that the infinite version of \ref{basic1} in Proposition~\ref{basic} is not always true.
For reasons to be explained below, we shall say that a locale $L$ is \emph{\donotbreakdash{$\bot$}scattered} if the nucleus $(-)^{**}\colon L\longrightarrow L$ preserves arbitrary meets, i.e.
  \begin{equation}\label{botscattered} \tbigwedge a_{i}^{**}\le(\tbigwedge a_{i})^{**},\qquad \{a_i\}_{i\in I}\subseteq L,\tag{$\bot$-scattered} \end{equation}

 \begin{remarks}
 (1) This terminology was introduced in \cite{SCAT} in the broader topos-theoretic setting. It is a property weaker than scatteredness, as introduced by Plewe in \cite{HORD}, where a locale is said to be \emph{scattered} if for each sublocale $S$ of $L$, the Booleanization $B_S$ is an open sublocale in $S$ (cf. Proposition~\ref{morescattered} below). More precisely, scatteredness is just the hereditary variant of \donotbreakdash{$\bot$}scatteredness.\\[2mm]
 (2) This notion has been recently considered by T. Dube and M.R. Sarpoushi under the name of \emph{near Booleaness} (cf. \cite[Theorem 4.9]{DS}).
\end{remarks}

We have the following easy characterization:

\begin{proposition}\label{wedgep} The following conditions are equivalent for a locale $L$\textup:
\begin{enumerate}[label={\rm(\roman*)}]
 \item \label{wedgep1} $L$ is \donotbreakdash{$\bot$}scattered\textup;
 \item \label{wedgep2} The nucleus $(-)^{**}\colon L\to L$ preserves arbitrary meets\textup;
 \item \label{wedgep3} The frame homomorphism $(-)^{**}\colon L\to B_L$ preserves arbitrary meets\textup;
 \item \label{wedgep4} $(-)^{**}\colon L\to B_L$ is a complete Heyting algebra homomorphism;
\item \label{wedgep5} $(\tbigwedge a_{i})^{*}\leq (\tbigvee a_i^{*})^{**}$ for all $\{a_i\}_{i\in I}\subseteq L$\textup;
 \item \label{wedgep6}  If $\{a_i\}_{i=1}^n\subseteq L$ satisfies $\tbigwedge a_{i}=0$ then $\tbigwedge a_{i}^{**}=0$.
\end{enumerate}
\end{proposition}
\pagebreak
 We turn now our attention to extremal disconnectedness in frames. We also have the following (see \cite{PJ,PP12}):

\begin{proposition}\label{EDCAR}
The following conditions are equivalent for a locale $L$\textup:
\begin{enumerate}[label={\rm(\roman*)}]
\item \label{EDCAR1} $(\tbigwedge a_{i})^{*}\le\tbigvee a_i^{*}$ for all $\{a_i\}_{i=1}^n\subseteq L$\textup;\hfill \textup(Second De Morgan law\textup)
 \item \label{EDCAR2} If $\{a_i\}_{i=1}^n\subseteq L$ satisfies $\tbigwedge a_{i}=0$ then $\tbigvee a_{i}^{*}=1$;
 \item \label{EDCAR3} $(\tbigvee a_{i})^{**}\le\tbigvee a_{i}^{**}$ for all $\{a_i\}_{i=1}^n\subseteq L$\textup;
\item \label{EDCAR4} The nucleus $(-)^{**}\colon L\longrightarrow L$  preserves finite joins\textup;
\item \label{EDCAR5} The nucleus $(-)^{**}\colon L\longrightarrow L$  is a lattice homomorphism\textup;
\item \label{EDCAR6} $(\tbigwedge a_{i})^{*}\le\tbigvee a_i^{*}$ for all $\{a_i\}_{i=1}^n\subseteq B_L$\textup;
 \item \label{EDCAR7}  If $\{a_i\}_{i=1}^n\subseteq L$ satisfies $(\tbigvee a_{i})^*=0$ then $\tbigvee a_{i}^{**}=1$.
\end{enumerate}
\end{proposition}

 Indeed, it is easy to check that \ref{EDCAR1}$\iff$\ref{EDCAR2}$\implies$\ref{EDCAR3}$\iff$\ref{EDCAR4}$\iff$\ref{EDCAR5}$\iff$\ref{EDCAR6}$\iff$\ref{EDCAR7}, and \ref{EDCAR7}$\implies$\ref{EDCAR2} is also true because of Lemma~\ref{basic}\,\ref{basic1}.

Note that we actually have the equality in \ref{EDCAR1}, \ref{EDCAR3}, \ref{EDCAR6} and \ref{EDCAR7}.  A locale satisfying any of the conditions of the proposition above is called \emph{extremally disconnected} or \emph{De Morgan}. Subsequently, we shall use ED as a shorthand for an extremally disconnected locale. Clearly, if $X$ is a topological space, $X$ is extremally disconnected in the usual sense if and only if the locale $\Omega(X)$ is extremally disconnected.

The main motivation for this paper is to study the infinite versions of the conditions in Proposition~\ref{EDCAR}.

\section{Infinite versions of extremal disconnectedness}\label{sec3}

\subsection*{Infinitely De Morgan locales}
We shall say that a locale $L$ is \emph{infinitely De Morgan} if it satisfies the \emph{infinite} second De Morgan law, i.e. if
 \begin{gather}\label{IDM} (\tbigwedge a_{i})^{*}\le\tbigvee a_i^{*},\qquad \{a_i\}_{i\in I}\subseteq L.\tag{IDM} \end{gather}
 For brevity such a locale will be referred to as an \emph{IDM} locale. (Note again that  we actually have the equality in (\ref{IDM})).

We have the following characterization (cf. Proposition~\ref{EDCAR}\,\ref{EDCAR1}--\ref{EDCAR2}):

\begin{proposition}\label{eqq1} The following conditions are equivalent for a locale $L$\textup:
\begin{enumerate}[label={\rm(\roman*)}]
\item \label{eqq11} $L$ is an IDM locale\textup;
\item \label{eqq12} If $\{a_i\}_{i\in I}\subseteq L$ satisfies $\tbigwedge a_{i}=0$ then $\tbigvee a_{i}^{*}=1$.
\end{enumerate}
\end{proposition}

\begin{proof} \ref{eqq11}$\implies$\ref{eqq12} is obvious. \ref{eqq12}$\implies$\ref{eqq11}: Let $\{a_i\}_{i\in I}\subseteq L$ and $a=(\tbigwedge a_i)^*$. Since $(\tbigwedge a_i)\wedge a=0$, it follows that $(\tbigvee a_{i}^{*})\vee a^*=1$. Hence, $a\le \tbigvee a_{i}^{*}$.
\end{proof}

If we restrict to the spatial case, we obtain the following characterization:

\begin{corollary}
A topological space $X$ is IDM if and only if for each family of closed sets with dense union, the family of their interiors covers $X$.
\end{corollary}

\begin{remarks}\label{remarks3.3} (1) A frame which is also a coframe does not necessarily satisfy (\ref{IDM}); one has an infinite second De Morgan law for supplements, but these need not coincide with pseudocomplements. For instance the locale $L=[0,1]$ is totally ordered (and thus a coframe) but not IDM, see (2) below.\\[2mm]
(2) It follows immediately that if $0$ is completely prime in a locale $L$ (i.e. if for each arbitrary family $\{a_i\}_{i\in I}\subseteq L$, $\tbigwedge a_i =0$ implies $a_i=0$ for some $i\in I$) then $L$ is an IDM locale. \\
The converse is true for linear locales (i.e. those which are chains). Indeed, let $L$ be a chain, then we have $a^{*}=0$ whenever $a\ne 0$ (and $0^{*}=1$), so let $\tbigwedge a_i =0$ for $\{a_i\}_{i\in I}\subseteq L$. By (\ref{IDM}), we have $\tbigvee a_{i}^{*}=1$, and thus there is some $i\in I$ with $a_{i}=0$.\\[2mm]
(3) Let $L$ be \emph{any} locale and define $L^*$ to be the poset obtained by adding a new bottom element $\bot$ to $L$ (cf. \cite[p.~504]{FS}). It is easily seen that $L^*$ is also a locale and that $\bot$ is completely prime. Accordingly, one has that this new locale is IDM.\\[2mm]
 (4) Any complete Boolean algebra is an IDM locale, but there are non Boolean IDM locales. For instance, any non Boolean locale $L$ such that $0$ is completely prime. An easy such example is the linear locale $L=\mathbb{N}\cup\{0,+\infty\}$, or any locale constructed as in (3) above.
\end{remarks}

 IDM locales are very close to being Boolean; in fact, under the very weak separation axiom of weak subfitness, both concepts coincide.
Recall that a frame is called \emph{weakly subfit} \cite{WS} if for each $a\ne 0$ there is some $c\ne 1$ with $c\vee a=1$. This property can also be characterized  by the following formula for pseudocomplementation:

\begin{lemma}\textup{(\cite[Theorem 5.2]{SUB})}
Let $L$ be a locale. The formula
\[a^*=\tbigwedge \{c\in L\mid c\vee a=1\}\quad\textrm{for all } a\in L\]
 is valid if and only if $L$ is weakly subfit.
\end{lemma}

Any Boolean algebra is trivially weakly subfit.  Moreover:

\begin{fact}\label{sbb} Let $L$ be a locale. Then $L$ is a complete Boolean algebra if and only if it is a weakly subfit and IDM locale.
\end{fact}

\begin{proof} We only need to prove sufficiency. Let $L$ be a weakly subfit and IDM locale and $a\in L$. By the previous lemma and the infinite second De Morgan law we get
\[a^{**} = ( \tbigwedge \{c\in L\mid c\vee a=1\})^*\le \tbigvee \{c^* \mid c\vee a=1\}.\]
Now if $c\vee a=1$ it follows that $c^*\le a$, hence $a^{**}\leq a$ for all $a\in L$. Thus $L$ is Boolean.
\end{proof}

\begin{remark}
 IDM does not imply weak subfitness and conversely. Indeed, the frame $L=\mathbb{N}\cup\{0,+\infty\}$ is IDM but not weakly subfit, and the cofinite topology on an infinite set is weakly subfit and not IDM.
\end{remark}

The situation is very different to the case of extremally disconnected spaces, where the Hausdorff axiom not only does not imply discreteness, but arguably the most important extremally disconnected spaces are Hausdorff (since it is in presence of this axiom that extremal disconnectedness has \emph{something to do} with connectedness).

\subsection*{Infinitely extremally disconnected locales}
We shall say that a locale $L$ is \emph{infinitely extremally disconnected} if the nucleus $(-)^{**}\colon L\longrightarrow L$  preserves arbitrary joins, i.e. if
 \begin{gather}\label{IED} (\tbigvee a_{i})^{**}\le\tbigvee a_{i}^{**},\qquad \{a_i\}_{i\in I}\subseteq L,\tag{IED} \end{gather}
 For brevity such a locale will be referred to as an \emph{IED} locale. (Note again that  we actually have the equality sign in (\ref{IED})).

We have the following characterization (cf. Proposition~\ref{EDCAR}\,\ref{EDCAR3}--\ref{EDCAR7}):

\begin{proposition}\label{eqq2} The following conditions are equivalent for a locale $L$\textup:
\begin{enumerate}[label={\rm(\roman*)}]
\item \label{eqq21} $L$ is an IED locale\textup;
\item \label{eqq22} The nucleus $(-)^{**}\colon L\longrightarrow L$  preserves arbitrary joins\textup;
\item \label{eqq23} The nucleus $(-)^{**}\colon L\longrightarrow L$ is a frame homomorphism\textup;
\item \label{eqq24} $(\tbigwedge a_{i})^{*}\le\tbigvee a_i^{*}$ for all $\{a_i\}_{i\in I}\subseteq B_L$\textup;
\item \label{eqq25} If $\{a_i\}_{i\in I}\subseteq L$ satisfies $(\tbigvee a_{i})^*=0$ then $\tbigvee a_{i}^{**}=1$.
\end{enumerate}
\end{proposition}

\begin{proof}
  \ref{eqq21}$\iff$\ref{eqq22}$\iff$\ref{eqq23}$\implies$\ref{eqq24} are obvious.\\[2mm]
 \ref{eqq24}$\implies$\ref{eqq25}: Let $\{a_i\}_{i\in I}\subseteq L$ such that $(\tbigvee a_{i})^*=0$. Then $\{a_i^*\}_{i\in I}\subseteq B_L$ and so $1=(\tbigvee a_{i})^{**}=(\tbigwedge a_{i}^*)^{*}\le\tbigvee a_i^{**}$.\\[2mm]
 \ref{eqq25}$\implies$\ref{eqq21}: Let $\{a_i\}_{i\in I}\subseteq L$ and $a=(\tbigvee a_{i})^{*}$. Since $(a\vee (\tbigvee a_{i}))^*=0$, it follows that $a^{**}\vee(\tbigvee a_{i}^{**}) =1$. Hence, $a^{*}=(\tbigvee a_{i})^{**}\le \tbigvee a_{i}^{**}$.
\end{proof}

If we restrict to the spatial case, we obtain the following characterization:

\begin{corollary}
A space $X$ is IED if and only if for each family of open sets with dense union, the family of the interiors of their closures covers $X$.
\end{corollary}

\begin{remarks} \label{remarks3.9} (1) For a finite locale $L$, it follows from Proposition~\ref{EDCAR} that all conditions in Propositions~\ref{eqq1} and \ref{eqq2} are equivalent. Consequently, in what follows we shall restrict our attention to infinite locales.\\[2mm]
 (2) In any irreducible (or hyperconnected) locale $L$ i.e. such that $B_L=\{0,1\}$, or equivalently such that $0$ is prime (cf. \cite{IRRED}) condition \ref{eqq24} in Proposition~\ref{eqq2} is trivially satisfied. Consequently, any irreducible locale $L$ is IED. In particular, linear locales are clearly irreducible, and hence IED.\\[2mm]
 (3) Since (\ref{IDM}) trivially implies condition \ref{eqq24} in Proposition~\ref{eqq2}, it follows that any IDM locale is IED, but there are IED locales which fail to be IDM.
 An easy such example is the cofinite topology on an infinite set.\\[2mm]
(4) Any IED locale is obviously extremally disconnected. However the converse is false, as any non-Boolean regular extremally disconnected locale shows (see Fact~\ref{ori} below). An easy such example is the Stone-\v Cech compactification of the frame of natural numbers.\\[2mm]
(5) In any semi-irreducible locale $L$ i.e. such that $B_L$ is finite, condition \ref{eqq24} in Proposition~\ref{eqq2} is clearly satisfied if the locale is ED. Consequently, any semi-irreducible extremally disconnected locale $L$ is IED.
 \end{remarks}

  Consequently we  have the following chain of implications
   \[\text{(CBA)}\implies \text{(\ref{IDM})}\implies \text{(\ref{IED})}\implies\text{(\ref{edframebat})},\]
   and none of them can be reversed.\pagebreak
Any IDM locale is trivially \donotbreakdash{$\bot$}scattered (cf. Proposition~\ref{eqq1} and Proposition~\ref{wedgep}\,\ref{wedgep6}). Moreover:

\begin{fact} \label{fact3.10} Let $L$ be a locale. Then $L$ is IDM if and only if it is \donotbreakdash{$\bot$}scattered and IED.
\end{fact}

\begin{proof}
We only need to prove sufficiency. Let $L$ be a \donotbreakdash{$\bot$}scattered and IED frame and consider $\{a_i\}_{i\in I}\subseteq L$ such that ${\tbigwedge a_{i}=0}$. By \donotbreakdash{$\bot$}scatteredness one has $(\tbigvee a_{i}^*)^*=\tbigwedge a_i^{**}\le(\tbigwedge a_i)^{**}=0$ and hence Proposition~\ref{eqq2}\,\ref{eqq25} implies that $\tbigvee a_{i}^{*}={\tbigvee a_{i}^{***}=1}$. By Proposition~\ref{eqq1}\,\ref{eqq12} $L$ is IDM.
\end{proof}

\begin{remark}
 IED does not imply and neither is implied by \donotbreakdash{$\bot$}scatteredness. Indeed, the cofinite topology on an infinite set is IED and not \donotbreakdash{$\bot$}scattered, and the frame of scattered real numbers (i.e. the topology generated by the usual topology of the real line and by arbitrary subsets of the irrationals) is \donotbreakdash{$\bot$}scattered (cf. Proposition~\ref{propspatialscat} below) but not IED.
\end{remark}

Now, if we combine this characterization with Propositions~\ref{wedgep} and \ref{eqq2} we obtain:

\begin{corollary} A locale is IDM if and only if the nucleus $(-)^{**}\colon L\to L$ is a complete Heyting algebra homomorphism.
\end{corollary}

Recall that the conjunction of weak subfitness and (\ref{IDM}) is equivalent to (CBA). One may also wonder which condition together with (\ref{IED}) implies Booleaness.

We first note that Fact~\ref{sbb} does not remain valid if we replace IDM by IED. For example, an infinite set with the cofinite topology is  a $T_1$-space (thus its associated locale is subfit) and, as we have observed before, it is also IED; though not, of course, discrete.

We recall that a locale $L$ is said to be \emph{semiregular} if every element is a join of regular elements (i.e. elements contained in $B_L$). A \emph{regular} locale is one in which $a=\tbigvee\{b\mid b^*\vee a=1\}$ for each $a\in L$. Every regular locale is semiregular and every zero-dimensional locale is regular (cf. \cite[III]{STONE}). 

A Boolean algebra is trivially zero-dimensional, and therefore it is regular (and semiregular). Moreover:

\begin{fact}\label{ori}
Let $L$ be a locale. Then $L$ is a complete Boolean algebra if and only if it is a semiregular IED locale.
\end{fact}

\begin{proof} Semiregularity means that $B_L$ generates $L$ by joins. But (\ref{IED}) is equivalent to $B_L$ being closed under joins. Hence $B_L=L$.
\end{proof}

\begin{remark}\label{rem3.15}
Locales $L$ for which $\mathcal{S}(L)^{op}$ is extremally disconnected were characterized by Plewe in \cite{SLAT}. It is therefore  natural to ask what the properties IDM and IED mean for locales of the form $\mathcal{S}(L)^{op}$. However, since $\mathcal{S}(L)^{op}$ is always a zero-dimensional locale, the previous fact implies that IED and IDM are in this case equivalent to Booleaness.
\end{remark}

Finally, we provide a condition under which extremal disconnectedness implies (\ref{IED}), due to Tomasz Kubiak in a private communication.
If $L$ is a locale, let us recall that a family  $\{a_i\}_{i\in I}\subseteq L$ is said to be \emph{locally finite} if there is a cover $\{b_j\}_{j\in J}$ such that for each $j\in J$ one has that  $a_i\wedge b_j=0$ for all but finitely many $i\in I$. The cover $\{b_j\}$ is said to \emph{witness} local finiteness of the family $\{a_i\}$.

\begin{proposition}\label{locfinite}
Let $L$ be a locale whose Booleanization $B_L$ is a locally finite family. Then \textup{(\ref{edframebat})} implies \textup{(\ref{IED})}.
\end{proposition}

\begin{proof}
Assume that $L$ is ED and $B_L$ is locally finite. Let $\{a_i\}_{i\in I}\subseteq L$. By an application of the first De Morgan law, we see that $( \tbigvee{a_i}^{**})^{**}= (\tbigvee{a_i})^{**}$. Now, let $\{b_j\}_{j\in J}$ be the cover witnessing local finiteness of $B_L$. For each $j\in J$, one has $$b_j  \wedge \left( \tbigvee{a_i}\right)^{**}=b_j  \wedge \left( \tbigvee{a_i}^{**}\right)^{**}\leq b_{j}^{**}\wedge \left( \tbigvee{a_i}^{**}\right)^{**}= \left( \tbigvee (b_j\wedge a_{i}^{**})\right)^{**}.$$
Since $\{a_{i}^{**}\}_{i\in I}\subseteq B_L$, there exists a finite $F_j\subseteq I$ such that $b_j\wedge a_{i}^{**}=0$ for all $i\not\in F_j$ and so we have
\[b_j  \wedge \bigl( \tbigvee{a_i}\bigr)^{**} \leq \Bigl( \tbigvee_{i\in F_j} (b_j\wedge a_{i}^{**})    \Bigr)^{**} = \tbigvee_{i\in F_j} b_{j}^{**}\wedge a_{i}^{**}\]
where the last equality follows from extremal disconnectedness and the fact that it is a finite join. We thus  obtain  
\[b_j  \wedge ( \tbigvee{a_i})^{**}  \leq \tbigvee_{i\in F_j}  a_{i}^{**}\le \tbigvee a_{i}^{**},\]
and the conclusion follows by taking joins as $j\in J$.
\end{proof}

\begin{remark}
Since finiteness of a family obviously implies its local finiteness, it is clear that the assumption in the proposition above is weaker than semiregularity; hence Proposition~\ref{locfinite} can be seen as a generalization of Remark~\ref{remarks3.9}~(5).
\end{remark}

\section{Scatteredness} \label{sec4} Let us now digress a bit in order to understand \donotbreakdash{$\bot$}scatteredness better.

We first state some further equivalent formulations of this condition.

\begin{proposition}\label{morescattered} The following conditions are equivalent for a locale $L$\textup:
\begin{enumerate}[label={\rm(\roman*)}]
 \item \label{morescattered1} $L$ is \donotbreakdash{$\bot$}scattered\textup;
 \item \label{morescattered2} There exists an open \donotbreakdash{$\bot$}scattered dense sublocale\textup;
 \item \label{morescattered3} The Booleanization $B_L$ is an open sublocale\textup;
 \item \label{morescattered4} The interior of a dense sublocale is dense.
\end{enumerate}
\end{proposition}

\begin{proof} \ref{morescattered1} $\implies$\ref{morescattered2} is obvious.\\[2mm] 
\ref{morescattered2}$\implies$\ref{morescattered3}: If $S$ is any sublocale whatsoever we have $B_S=B_{\overline{S}}$. Indeed, since $S$ is dense in $\overline{S}$ and $B_{\overline S}$ is the least such sublocale of $\overline{S}$, one has $B_{\overline{S}}\subseteq S$. But now $B_{\overline{S}}$ is dense in $S$, whence it is a dense Boolean sublocale of $S$. Thus $B_S=B_{\overline{S}}$. If $S$ is dense and \donotbreakdash{$\bot$}scattered, one has that $B_S=B_L$ is open in $S$. If additionally $S$ is open, it follows that $B_L$ is open in $L$.\\[2mm]
\ref{morescattered3}$\implies$\ref{morescattered4}: If $S$ is a dense sublocale, we have $B_L \subseteq S$, and so by assumption, $B_L=\textrm{int}(B_L)\subseteq \textrm{int}(S)$. Hence $\textrm{int}(S)$ must be dense too.\\[2mm]
\ref{morescattered4}$\implies$\ref{morescattered1}: Since $B_L$  is dense, by assumption so is $\textrm{int}(B_L)$. But $B_L$ is the least dense sublocale, hence $B_L\subseteq \textrm{int}(B_L)$, i.e., $B_L$ is open. It follows that the frame homomorphism $(-)^{**}$ (left adjoint to the sublocale embedding $B_L\subseteq L$) is a complete Heyting algebra homomorphism, i.e. $L$ is \donotbreakdash{$\bot$}scattered.
\end{proof}

If we restrict to the spatial case, we obtain the following characterizations, already mentioned in \cite{BP,SCAT}:
\begin{proposition}\label{propspatialscat} The following conditions are equivalent for a $T_0$-space $X$\textup:
\begin{enumerate}[label={\rm(\roman*)}]
\item \label{propspatialscat1} $\Omega(X)$ is \donotbreakdash{$\bot$}scattered\textup;
\item \label{propspatialscat2} There exists an open, dense and discrete subset\textup;
\item \label{propspatialscat3} The set of isolated points is dense.
\end{enumerate}
\end{proposition}

\begin{proof}
\ref{propspatialscat1}$\implies$\ref{propspatialscat2}: If $B_{\Omega(X)}$ is open, it corresponds to an open dense subspace of $X$ whose frame of opens is Boolean. Now, this subspace with the induced topology is also $T_0$ and hence (by Booleaness of its frame of opens) it is discrete.\\[2mm]
\ref{propspatialscat2}$\implies$\ref{propspatialscat3}: If $D$ is open and discrete, it is contained in the set of isolated points. Hence, if $D$ is dense, so is the set of isolated points.\\[2mm]
\ref{propspatialscat3}$\implies$\ref{propspatialscat1}: The set of isolated points is trivially open and discrete. Hence, under the assumption, it is an open dense discrete subspace. It thus corresponds to an open dense and Boolean sublocale of $\Omega(X)$.
\end{proof}

The following results will be needed later on.

\begin{lemma}\label{scin}
The property of being \donotbreakdash{$\bot$}scattered is inherited  
\begin{enumerate}[label={\rm(\arabic*)}]
\item \label{scin1} by open sublocales\textup;
\item \label{scin2} by dense sublocales.
\end{enumerate}
\end{lemma}

\begin{proof}
\ref{scin1} Since $\mathfrak{o}(a)\cong \newdownarrow a$, it reduces to show that $\newdownarrow a$ is \donotbreakdash{$\bot$}scattered whenever $L$ is. Note that $\newdownarrow a\subset L$ is closed under arbitrary joins and arbitrary non-empty meets. It was shown in \cite[Remark~5.4\,(2)]{PJ} that the pseudocomplement $b^{*a}$ of an element $b$ in $\newdownarrow a$ is given by $b^{*a}=b^{*}\wedge a$ and that $(b^{*a})^{*a}=b^{**}\wedge a$. Combining both remarks it is clear that the assertion holds.\\[2mm]
\ref{scin2} Pseudocomplements in dense sublocales are the same as in the ambient locale; and  sublocales are always closed under meets. It is then obvious that being \donotbreakdash{$\bot$}scattered is inherited by dense sublocales.
\end{proof}

\begin{lemma}\label{scOp}
If $f\colon L\longrightarrow M$ is an open localic map and $L$ is \donotbreakdash{$\bot$}scattered, then so is $f[L]$.
\end{lemma}

\begin{proof} First, since a localic map is open if and only if both halves of its \donotbreakdash{surjection}embedding factorization are open, we can assume that $f$ is also a surjection. Now, openness of a localic morphism implies that inverse image commutes with closure. It follows that $L=f^{-1}[M]=f^{-1}\left[\overline{B_M}\right]=\overline{f^{-1}[B_M]}$,  and since $B_L$ is the least dense sublocale, one has $B_L\subseteq f^{-1}[B_M]$, which by adjunction is equivalent to $f[B_L]\subseteq B_M$.  But $f$ is surjective, so in particular $0=f(0)\in f[B_L]$, and thus $f[B_L]$ is dense in $M$. Hence we have the equality $B_M=f[B_L]$. If $L$ is \donotbreakdash{$\bot$}scattered, $B_L$ is open and thus so is $B_M$ by openness of $f$.
\end{proof}

\begin{remark}
One may consider the notion of $\bot$-scatteredness when restricted to locales of the form $\mathcal{S}(L)^{op}$ (cf. Remark \ref{rem3.15} above). This will be carried out in a further work concerning sublocale lattices. 
\end{remark}

\section{Properties of IDM and IED locales.}\label{sec5}
Let us come back to the main topic of the paper. We now state some further equivalent formulations of these properties in terms of the Booleanization:

\begin{proposition} \label{veep} The following conditions are equivalent for a locale $L$\textup:
\begin{enumerate}[label={\rm(\roman*)}]
\item \label{veep1} $L$ is an IED locale\textup;
\item \label{veep2} The Booleanization $B_L$ is a subframe of $L$.
\end{enumerate}
\end{proposition}

\begin{proof} The result follows immediately from Proposition~\ref{eqq2} and the fact that a nucleus preserves arbitrary joins if and only if its associated sublocale is closed under arbitrary joins.
\end{proof}

\begin{proposition}\label{idmcar}
The following conditions are equivalent for a locale $L$\textup:
\begin{enumerate}[label={\rm(\roman*)}]
\item \label{idmcar1} $L$ is an IDM locale\textup;
\item \label{idmcar2} The Booleanization $B_L$ is an open sublocale and a subframe of $L$\textup;
\item \label{idmcar3} The Booleanization $B_L$ is open and a complete sublattice of $L$.
\end{enumerate}
\end{proposition}

\begin{proof}
Immediate from Propositions~\ref{morescattered} and \ref{idmcar} and Fact~\ref{fact3.10}.
\end{proof}

It is a well-known fact that extremal disconnectedness is preserved under taking closed or dense sublocales \cite{HERPROP} and taking images  under open localic morphisms \cite{PARAL}. Next results generalizes those facts to the infinite setting.

\begin{proposition} \label{herr} Both properties IED and IDM are inherited
\begin{enumerate}[label={\rm(\arabic*)}]
\item \label{herr1} by open sublocales\textup;
\item \label{herr2} by dense sublocales.
\end{enumerate}
\end{proposition}

\begin{proof} Clearly, the assertion for IDM locales  will follow from the one for IED locales combined with Lemma~\ref{scin} and Fact~\ref{fact3.10}. Now, that IED is inherited by open sublocales can be proved as in Lemma~\ref{scin}\,\ref{scin1}. Let us finally show that IED is hereditary with respect to dense sublocales. Let $S$ be a dense sublocale of an IED locale $L$ and denote by $\tbigsqcup$ joins in $S$. Note that in any dense sublocale one has $(\tbigvee s_i)^* =( \tbigsqcup s_i)^*$ for each  $\{s_i\}_{i\in I}\subseteq S$. Indeed, by the first de Morgan law in $L$ (resp. in $S$) and the fact that pseudocomplements and meets are the same in $S$ and $L$, both sides are equal to $\tbigwedge s_{i}^*$.
Since $S$ is dense, we have that $B_L\subseteq S$, and therefore by the (\ref{IED}) law in $L$, one has $\tbigvee s_{i}^{**}=(\tbigvee s_i)^{**}\in S$. Thus the join of $\{s_{i}^{**}\}_{i\in I}\subseteq S$ in $S$ coincides with the one in $L$. It  follows that $\tbigsqcup s_{i}^{**} =\tbigvee s_{i}^{**}=( \tbigvee s_{i})^{**}=( \tbigsqcup s_{i})^{**}.$
\end{proof}

We now have the following trivial observation:

\begin{lemma}
Let $M$ be a subframe of $L$ and assume that $M$ is closed under pseudocomplementation in $L$. If $L$ is IED, then so is $M$.
\end{lemma}

\begin{remark}\label{nearlyopen}
 A frame homomorphism $h$ is said to be \emph{nearly open} \cite{BP} if it commutes with pseudocomplementation (equivalently, if it commutes with double pseudocomplementation).  The previous lemma can therefore be stated as: 
\begin{quote}{\em If $L$ is IED and $M\hookrightarrow L$ is a nearly open subframe embedding, then $M$ is also IED.}
\end{quote}
\end{remark}
\pagebreak
\begin{corollary}
If $f\colon L\longrightarrow M$ is an open localic map and $L$ is IED \textup(resp. IDM\textup), then so is $f[L]$.
\end{corollary}

\begin{proof}
In view of Lemma~\ref{scOp} and Fact~\ref{fact3.10}, it suffices to show the assertion for IED. Moreover, since a localic map is open if and only if both halves of its surjection-embedding factorization are open, one can assume without loss of generality that $f$ is surjective. Now, the left adjoint $f^*$ of $f$ is (isomorphic to) an open subframe inclusion, so the result follows from the previous lemma and the fact that an open frame homomorphism is nearly open.
\end{proof}

\section{Hereditary variants}\label{sec6}
Given a property $P$ of locales, a locale $L$ is said to be \emph{hereditarily} $P$ if each sublocale of $L$   satisfies $P$. Our main interest in this section is to study hereditarily IDM and hereditarily IED locales.

We first note the following:

\begin{proposition}
 Let $P$ be a property of locales such that each dense sublocale of a locale satisfying $P$ also satisfies $P$. Then a locale $L$ is hereditarily $P$ if and only if each closed sublocale of $L$ satisfies $P$.
\end{proposition}

\begin{proof} We only need to prove sufficiency. Let  $L$ be a locale such that each closed sublocale of $L$ satisfies $P$ and let $S$ be an arbitrary sublocale of $L$. Then $\overline{S}$ is closed and so it has property $P$. Now $S$ is dense in $\overline{S}$ and so by dense-heredity of $P$, it follows that $S$  also has property $P$.
\end{proof}

From Lemma~\ref{scin} and Proposition~\ref{herr} we get then the following (note that scatteredness is precisely hereditary \donotbreakdash{$\bot$}scatteredness --- cf. \cite{HORD}):

\begin{corollary}\label{closedhereditary}
 Let $L$ be a locale. Then:
\begin{enumerate}[label={\rm(\arabic*)}]
\item \label{closedhereditary1} $L$ is scattered if and only if each closed sublocale of $L$ is \donotbreakdash{$\bot$}scattered.
\item \label{closedhereditary2} $L$ is hereditarily extremally disconnected if and only if each closed sublocale of $L$ is extremally disconnected.
\item \label{closedhereditary3} $L$ is hereditarily IDM if and only if each closed sublocale of $L$ is IDM.
\item \label{closedhereditary4} $L$ is hereditarily IED if and only if each closed sublocale of $L$ is IED.
\end{enumerate}
\end{corollary}

\begin{remark}
Since a spatial locale can have more sublocales than subspaces, it is not immediately clear whether hereditary IED and IDM are conservative properties. However, in view of \ref{closedhereditary3} and \ref{closedhereditary4} above, it follows that they are indeed conservative (closed sublocales are induced by the corresponding closed subspaces).
\end{remark}

The following proposition summarizes several well-known characterizations of hereditarily extremally disconnected locales, see for example \cite{HERPROP,PJ}.

\begin{proposition}
The following conditions are equivalent for a locale $L$\textup:
\begin{enumerate}[label={\rm(\roman*)}]
\item $L$ is hereditarily extremally disconnected\textup;
\item $(a\to b)\vee (b\to a)=1$ for all $a,b\in L$\textup;\hfill \textup(Strong De Morgan law\textup)
\item $( \tbigwedge a_i)\to b\le\tbigvee (a_i\to b)$ for all $b\in L$ and all $\{a_i\}_{i=1}^n\subseteq L$\textup;
\item $( \tbigwedge a_i)\to b\le\tbigvee (a_i\to b)$ for all $b\in L$ and all $\{a_i\}_{i=1}^n\subseteq \mathfrak{c}(b)$\textup;
\item $(( \tbigvee a_i)\to b)\to b\le\tbigvee( (a_i\to b)\to b)$ for all $b\in L$ and all $\{a_i\}_{i=1}^n\subseteq L$\textup;
\item $(( \tbigvee a_i)\to b)\to b\le\tbigvee( (a_i\to b)\to b)$  for all $b\in L$ and all $\{a_i\}_{i=1}^n\subseteq \mathfrak{c}(b)$.
\end{enumerate}
\end{proposition}

We now have the following characterizations of hereditarily IDM and IED locales:

\begin{proposition} \label{hIDM}
The following conditions are equivalent for a locale $L$\textup:
\begin{enumerate}[label={\rm(\roman*)}]
\item \label{hIDM1} $L$ is hereditarily IDM\textup;
\item \label{hIDM2} $( \tbigwedge a_i)\to b\le\tbigvee  (a_i\to b)$ for all $b\in L$ and all $\{a_i\}_{i\in I}\subseteq L$\textup; \hfill\textup($\tbigwedge\tbigvee$-distr\textup)
\item \label{hIDM3} $( \tbigwedge a_i)\to b\le \tbigvee (a_i\to b)$ for all $b\in L$ and all $\{a_i\}_{i\in I}\subseteq \mathfrak{c}(b)$.
\end{enumerate}
\end{proposition}

\begin{proof}
 \ref{hIDM1}$\implies$\ref{hIDM2}: Let $b\in L$, $\{a_i\}_{i\in I}\subseteq L$ and $d=(\tbigwedge a_i)\wedge b$. By hypothesis $\mathfrak{c}(d)$ is IDM, and since pseudocomplementation in $\mathfrak{c}(d)$ is given by $x^{* \mathfrak{c}(d)}=x\to d$ and $\mathfrak{c}(d)$ is closed under meets and nonempty joins in $L$ it follows that
$(\tbigwedge a_{i})\to d=(\tbigwedge a_{i})^{*\mathfrak{c}(d)}\le\tbigvee a_i^{*\mathfrak{c}(d)}=\tbigvee (a_i\to d)$. Consequently
\[( \tbigwedge a_i)\to b=(\tbigwedge a_{i})\to d\le\tbigvee (a_i\to d)\le\tbigvee  (a_i\to b).\]
\ref{hIDM2}$\implies$\ref{hIDM3} is obvious.\\[2mm]
\ref{hIDM3}$\implies$\ref{hIDM1}: By Corollary~\ref{closedhereditary}, it is enough to prove that each closed sublocale is IDM. Let $b\in L$ and $\{a_i\}_{i\in I}\subseteq \mathfrak{c}(b)$ such that $\tbigwedge a_{i}=b$. Then
\[\tbigvee a_{i}^{*\mathfrak{c}(b)}=\tbigvee (a_{i}\to b)=( \tbigwedge a_i)\to b=1.\]
By Proposition~\ref{eqq1}\,\ref{eqq12} it follows that $\mathfrak{c}(b)$ is IDM.
\end{proof}

\begin{proposition} \label{hIED}
The following conditions are equivalent for a locale $L$\textup:
\begin{enumerate}[label={\rm(\roman*)}]
\item \label{hIED1} $L$ is hereditarily IED\textup;
\item \label{hIED2} $(( \tbigvee a_i)\to b)\to b\le\tbigvee ((a_i\to b)\to b)$ for all $b\in L$ and all $\{a_i\}_{i\in I}\subseteq L$\textup;
\item \label{hIED3} $(( \tbigvee a_i)\to b)\to b\le\tbigvee ((a_i\to b)\to b)$  for all $b\in L$ and all $\{a_i\}_{i\in I}\subseteq \mathfrak{c}(b)$\textup.
\end{enumerate}
\end{proposition}

\begin{proof}
\ref{hIED1}$\implies$\ref{hIED2}: Let $b\in L$, $\{a_i\}_{i\in I}\subseteq L$ and $b_i=a_i\vee b$ for each $i\in I$. By hypothesis $\mathfrak{c}(b)$ is IED, and so
$$((\tbigvee b_{i})\to b)\to b=(\tbigvee b_{i})^{**\mathfrak{c}(b)}\le\tbigvee b_i^{**\mathfrak{c}(b)}=\tbigvee ((b_i\to b)\to b).$$ Since $a\to b=(a\vee b)\to b$ or each $a\in L$, it follows that
\begin{align*}
(( \tbigvee a_i)\to b)\to b& =((\tbigvee b_{i})\to b)\to b\le  \tbigvee ((b_i\to b)\to b)=\tbigvee ((a_i\to b)\to b).
\end{align*}
\ref{hIED2}$\implies$\ref{hIED3} is obvious.\\[2mm]
\ref{hIED3}$\implies$\ref{hIED1}: By Corollary~\ref{closedhereditary}, it is enough to prove that each closed sublocale is IED. Let $b\in L$ and $\{a_i\}_{i\in I}\subseteq \mathfrak{c}(b)$ such that $(\tbigvee a_{i})^{*\mathfrak{c}(b)}=b$. Then
\begin{align*}
1&=(\tbigvee a_{i})^{*\mathfrak{c}(b)}\to b=(( \tbigvee a_i)\to b)\to b\le\tbigvee ((a_i\to b)\to b)=\tbigvee a_{i}^{**\mathfrak{c}(b)}.
\end{align*}
By Proposition~\ref{eqq2}\,\ref{eqq25} it follows that $\mathfrak{c}(b)$ is IDM.
\end{proof}

Note once again that all inequalities in the statements of the three previous propositions are indeed equalities.

\begin{remark}
 By Fact~\ref{fact3.10} we now have that a locale $L$ is hereditarily IDM if and only if it is scattered and hereditarily IED.
\end{remark}

\begin{example}
(1) Every linear locale is hereditarily IED. Indeed, a sublocale of $L$ is also linear and thus IED by Remark~\ref{remarks3.9}\,(2).\\[2mm]
(2) For any locale $L$, the IDM locale $L^*$ constructed in Remark~\ref{remarks3.3}\,(2) is not, in general, hereditarily IDM nor hereditarily IED (in fact, $L^*$ is hereditarily IDM, resp. IED, if and only if so is $L$, because proper closed sublocales of both $L^*$ and $L$ coincide).
\end{example}

\section{Categorical aspects of infinite extremal disconnectedness}\label{sec7}
In what follows, the full subcategories of $\cat{Frm}$ consisting of IED frames and of complete Boolean algebras will be denoted by $\cat{IEDFrm}$ and $\cat{CBAlg}$ respectively. Note that morphisms in $\cat{CBAlg}$ are exactly complete Boolean homomorphisms.
\smallskip

Banaschewski and Pultr studied in \cite{BP} conditions under which the Booleanization construction behaves functorially in a natural manner, in the sense that for each frame homomorphism $f\colon L\to M$ there is a (unique) frame homomorphism $B(f)\colon B_L\to B_M$ such that the following square in \cat{Frm} commutes
\[\begin{tikzcd}
 L \arrow{r}{B(f)} \arrow{d}[swap]{(-)^{**}} &\arrow{d}{(-)^{**}} M\   \\
B_L  \arrow{r}{f} & B_M
\end{tikzcd}\]
It was shown there that the above condition is satisfied if and only if $f$ is \emph{weakly open} (sometimes also called \emph{skeletal}), i.e. if $f(a^{**})\leq f(a)^{**}$ for all $a\in L$.
This means that for any subcategory $\CC$ of $\cat{Frm}$ consisting only of weakly open morphisms and containing the category $\cat{CBAlg}$ there is a natural transformation $\eta\colon 1_\CC \longrightarrow I\circ B$ where $B\colon \cat{IEDFrm}\longrightarrow \cat{CBAlg}$ denotes Booleanization and $I\colon \cat{CBAlg}\longrightarrow \cat{IEDFrm}$ denotes inclusion, and the components of $\eta$ are given by the double-negation morphisms (which are trivially weakly open). 

It was further pointed out in \cite[1.4]{BP} that we cannot hope to make $B$ functorial by restricting the class of objects rather than the class of morphisms, for if a \emph{full} subcategory $\CC$ satisfies the above condition then $\CC$ is necessarily the whole of the category of complete Boolean algebras.

The goal of this section is to propose an alternative setting for making the Booleanization functorial, still behaving naturally and restricting to a \emph{full} subcategory of \cat{Frm}; nevertheless changing the effect of $B$ in morphisms and the components of the natural transformation. More precisely, we consider the largest class of locales for which the inclusion $B_L\hookrightarrow L$ is a frame homomorphism --- the class of IED locales--- and show that the desired functoriality condition is then satisfied.

\begin{lemma}
Let $\mathcal{C}$  be a subcategory of $\cat{Frm}$ consisting of IED locales. Then the Booleanization extends to a functor $B\colon \mathcal{C}\longrightarrow \cat{CBAlg}$ naturally, that is, for each morphism $f\colon L\longrightarrow M$ in $\mathcal{C}$  we have a commutative diagram of frame homomorphisms
\[\begin{tikzcd}
 B_L \arrow{r}{B(f)} \arrow[hookrightarrow]{d}[swap]{i_L} &\arrow[hookrightarrow]{d}{i_M} B_M \   \\
L  \arrow{r}{f} & M
\end{tikzcd}\]
\end{lemma}

\begin{proof}
We noted in Proposition~\ref{veep} that for an IED locale $L$,  $B_L$ is a subframe of $L$, i.e. the inclusion $i_L\colon B_L\hookrightarrow L$ lies in $\cat{Frm}$. Let now $f\colon L\longrightarrow M$ be a frame \pagebreak homomorphism in $\CC$. Since IED implies extremal disconnectedness, $L$ is in particular extremally disconnected.  Now,  $f$ preserves complements, and so it sends regular(=complemented) elements to complemented elements. Therefore, $f$ restricts to a map $B(f)\colon B_L\longrightarrow B_M$. The restriction $B(f)$ is of course a frame homomorphism,  because $B_L$ and $B_M$ are subframes of $L$ resp. $M$. Hence all the maps in the square above are frame homomorphisms.
\end{proof}

\begin{proposition} The  subcategory $\cat{CBAlg}$ is a full \textup(mono\textup)coreflective category of $\cat{IEDFrm}$.
\end{proposition}

\begin{proof}
Since frame homomorphisms preserve complements, any frame homomorphism $f\colon B\longrightarrow L$ where $B$ is Boolean maps $B$ into $B_L$, and hence it factors uniquely through the map  $B_L\hookrightarrow L$ (the counit of the adjunction), which is a frame homomorphism provided $L$ is IED.
\end{proof}

In passing, we note that \cat{IEDFrm}  cannot be cocomplete because neither is $\cat{CBAlg}$.

\begin{remark}
It is, however, false that the Booleanization functor is a left adjoint, and hence, in particular, $B$ is not a reflection functor, despite the fact that the category \cat{CBAlg} is closed under limits in \cat{IEDFrm}  (indeed, limits in \cat{CBAlg} are computed as in \cat{Frm}, and it is thus clear that they are also limits in the full subcategory \cat{IEDFrm}).\\

 Let $\boldsymbol{1}=\{ 0=1 \}$ and $\boldsymbol{2}=\{0,1\}$ denote the one-element and two-element Boolean algebras respectively, namely the terminal and initial objects in $\cat{Frm}$. Assume by way of contradiction that $G\colon \cat{CBAlg}\longrightarrow $ \cat{IEDFrm} is a right adjoint of $B\colon$ \cat{IEDFrm} $\longrightarrow \cat{CBAlg}$. Clearly, $G$ is not the constant functor sending every object to $\boldsymbol{1}$,  so there is a Boolean algebra $B_0$ with $G(B_0)\ne \boldsymbol{1}$. Therefore, the unique frame homomorphism $\boldsymbol{2}\to G(B_0)$ is injective (because  $0\ne 1$ in $G(B_0)$), i.e. it is monic in $\cat{Frm}$. Since for each frame $L$ the representable $\cat{Frm}(L,-)$ preserves monos, it follows that we have an injective map $\cat{Frm}(L,\boldsymbol{2})\hookrightarrow \cat{Frm}(L,G(B_0))$, and by transposition across the adjunction (whenever $L$ is IED) the latter is an injective map $\cat{Pt}(L)\hookrightarrow \cat{Frm}(B_L,B_0)$. This is clearly not true in general, e.g. for $L=\Omega(X)$ for an irreducible sober space with more than one point, one has $\cat{Pt}(L)\cong X$ but $\cat{Frm}(B_L,B_0)=\cat{Frm}(\boldsymbol{2},B_0)=\{\star\}$.

This remark should be compared to \cite[Theorem~3.1]{BP}, which asserts that the Booleanization is (a left adjoint but) not a right adjoint.
\end{remark}

We conclude this section by exploring further the category of IED locales. The following proposition (together with the results thereafter) provides some evidence of the fact that the IED condition itself is actually a better behaved strengthening of extremal disconnectedness compared to the IDM condition. Furthermore, in view of Caramello's DeMorganization construction \cite[Theorem~2.10]{DM} (namely, the existence of the largest dense De Morgan sublocale) it also seems to share a stronger parallel with extremal disconnectedness.\pagebreak

Before proving the main result, we note the following: if $L$ is a locale and  $f,g\colon\mathcal{P}(L)\longrightarrow L$  any two mappings, then the subset 
\[S_{f,g}=\{a\in L\mid f(A)\to a=g(A)\to a \mbox{ for any } A\subseteq L\}\]
  is always a sublocale of $L$.\

\begin{proposition}\label{ldense}
For any locale there is the largest dense IED sublocale.
\end{proposition}

\begin{proof}
Let 
\[S=\{a\in L\mid (\tbigvee a_{i})^{**}\to a=(\tbigvee a_{i}^{**})\to a\mbox{ for every }\{a_i\}_{i\in I}\subseteq L\}.\]
By the comment before the statement,  $S$ is a sublocale of $L$; and an application of the first De Morgan law shows that $0\in S$, i.e. $S$ is a dense sublocale. If $T\subseteq L$ is an arbitrary dense IED sublocale of $L$, we want to show that $T\subseteq S$. By density, pseudocomplements coincide in each of the locales $T$,  $S$ and $L$. Let $\nu_T$ denote the left-adjoint to the sublocale embedding $\iota_T\colon T\to L$ and denote the joins in $T$ by $\sqcup$. Since $\nu_T$ is a dense surjection, it preserves pseudocomplements (i.e. it is nearly open). Let $t\in T$. Our goal is to show  that $t\in S$, that is, $\left(\tbigvee a_{i}\right)^{**} \to t = \left(\tbigvee a_{i}^{**}\right)\to t$ for any family $\{a_i\}_{i\in I}\subseteq L$. Since $T$ satisfies (\ref{IED}), one obtains $$\nu_T\left( \tbigvee a_{i}^{**} \right)=\tbigsqcup \nu_T(a_i)^{**}=\bigl( \tbigsqcup \nu_T(a_i)\bigr)^{**}=\bigl(\nu_T(\tbigvee a_i ) \bigr)^{**} = \nu_T\bigl( ( \tbigvee a_i )^{**}\bigr).$$ 
Then, since $t\in T$, we have
$$(\tbigvee a_{i})^{**} \to t =\nu_T\bigl((\tbigvee a_{i})^{**}  \bigr) \to t = \nu_T( \tbigvee a_{i}^{**} ) \to t =  (\tbigvee a_{i}^{**}) \to t,$$
as required. The only point remaining is to show that $S$ is IED.  Let $\{s_i\}_{i\in I}\subseteq S$. Then  
$$(\tbigsqcup s_i)^{**}=\bigl(\nu_{S}( \tbigvee s_{i})\bigr)^{**}  =\nu_{S}\bigl(  (\tbigvee s_{i})^{**} \bigr) = \nu_{S} ( \tbigvee s_{i}^{**})=\tbigsqcup s_{i}^{**},$$
where $\tbigsqcup$ now denotes join in $S$. This proves the result.
\end{proof}

\begin{remark}
The construction of the largest IED sublocale is not generally functorial  (this should not be a surprise because neither of the Booleanization or the DeMorganization construction \cite{DM} are normally functorial).  Nevertheless, there are certain  morphisms for which it is. We do not know how to characterize the class of those morphisms which restrict to the largest IED sublocales but it notably includes all the nearly open frame homomorphisms (cf. Remark~\ref{nearlyopen}).
\end{remark}

\begin{lemma}
Let $L$ be a linear locale. There exists the largest dense IDM sublocale if and only if $L$ is IDM.
\end{lemma}

\begin{proof}
The ``if'' part is trivial. Conversely,  in a linear locale $a\to b= 1$ or $b$ for each $a,b\in L$, and hence a sublocale is just a subset closed under meets. Moreover, since a sublocale set is in particular a subposet, it is also a chain, and hence the characterization in Remarks~\ref{remarks3.3}\,(2) still applies. Denote by $S$ the largest dense IDM sublocale. By contradiction, if $S\ne L$, pick $a\in L-S$ (where $-$ stands for set-theoretic difference). Then $S\cup\{a\}$ is obviously closed under meets and hence a (dense) sublocale. Furthermore, $0$ is completely prime in $S\cup\{a\}$, for if $a\wedge \tbigwedge a_i=0$ for some $\{a_i\}_{i\in I}\subseteq S$, since $a\ne 0$ and $0$ is always prime in a chain, it follows that $\tbigwedge a_i=0$ and so $a_i=0$ for some $i\in I$. This contradicts the maximality of $S$.
\end{proof}

The previous lemma yields examples of locales that do not possess  the largest dense IDM sublocale (for instance,  $L=[0,1]$).

Proposition~\ref{ldense} immediately yields a sufficient condition which ensures the existence of the largest dense IDM sublocale:

\begin{corollary}
For any \donotbreakdash{$\bot$}scattered locale there is the largest dense IDM sublocale.
\end{corollary}

\begin{proof}
The largest dense IED sublocale, whose existence is ensured by Proposition~\ref{ldense}, is also \donotbreakdash{$\bot$}scattered,  due to the fact that \donotbreakdash{$\bot$}scatteredness is hereditary with respect to dense sublocales.  Hence there is the largest dense IDM sublocale.
\end{proof}

The condition of having open Booleanization seems to be far from being necessary: as a counterexample, any non $\bot$-scattered fit locale $L$  works (of course, there are plenty of examples of those). Indeed, fitness is hereditary and so every sublocale of $L$ is fit. So IDM sublocales are just the same as Boolean sublocales by Fact~\ref{sbb}, thus the largest dense IDM sublocale exists (and coincides with the Booleanization of $L$).

\begin{remark}
Originally the DeMorganization construction was proved more generally for toposes, cf. \cite{DM,FIELDS}. Therefore, it seems natural to consider Proposition~\ref{ldense} in that context. We are not going to do so in this paper, except to say that one would need to define the IED property for toposes appropriately. It is not sensible to define an IED topos to be one in which double negation $\neg\neg\colon\Omega\longrightarrow \Omega$ has an internal right adjoint, since an easy modification of the proof of Theorem~6 in \cite{CCD} shows that in that case the topos is necessarily Boolean.
\end{remark}
\section*{Acknowledgements} The author is grateful to his PhD supervisors Javier Guti\'errez Garc\'ia and Jorge Picado for their help and for the many improvements they made to this paper.


\begin{thebibliography}{9}

\bibitem{BP}
B.~Banaschewski and A.~Pultr, \emph{Booleanization}, Cahiers Topologie G\'{e}om. Diff\'{e}rentielle Cat\'{e}g. \textbf{37} (1996), 41--60.

\bibitem{FS}
M.~Barr, \emph{Fuzzy set theory and topos theory}, Canad. Math. Bull. \textbf{29} (1986), 501--508.

\bibitem{HB}
F.~Borceux, \emph{Handbook of Categorical Algebra: Volume 3, Categories of sheaves}, Encyclopedia of Mathematics and its Applications, 52, Cambridge University Press, Cambridge, 1994.

\bibitem{DM}
O.~Caramello, \emph{De Morgan classifying toposes}. Adv. Math. \textbf{222} (2009), 2117--2144.

\bibitem{FIELDS}
O.~Caramello and P.~T.~Johnstone, \emph{De Morgan’s law and the theory of fields}. Adv. Math. \textbf{222} (2009), 2145--2152.

\bibitem{IRRED}
 T.~Dube, \emph{Irreducibility in pointfree topology}, Quaest. Math. \textbf{27} (2004), 231--241.

\bibitem{DS}T.~Dube and M.R.~Sarpoushi, \emph{On densely normal locales}, Topol. Appl., to appear (doi: 10.1016/j.topol.2019.107015).

\bibitem{SCAT}
L.~Esakia, M.~Jibladze and D.~Pataraia, \emph{Scattered toposes}, Ann. Pure Appl. Logic \textbf{103} (2000), 97--107.

\bibitem{GLIVENKO}
V.~Glivenko, \emph{Sur quelque points de la logique de M. Brouwer}, Bull. Acad. R. Belg. Cl. Sci.  \textbf{15} (1929), 183-188.

\bibitem{PG} J.~Guti\'errez Garc\'ia, T.~Kubiak and J.~Picado, \emph{Lower and upper regularizations of frame semicontinuous real functions,} Algebra Universalis \textbf{60} (2009), 169--184.

 \bibitem{HERPROP} J.~Guti\'errez Garc\'ia, T.~Kubiak and J.~Picado, \emph{On hereditary properties of extremally disconnected frames and normal frames}, Topology Appl., to appear (doi: 10.1016/j.topol.2019.106978).

 \bibitem{PARAL} J.~Guti\'{e}rrez Garc\'{\i}a and J.~Picado, \emph{On the parallel between normality and extremal disconnectedness}, J. Pure Appl. Algebra \textbf{218} (2014), 784--803.


\bibitem{WS}
 H.~Herrlich and A.~Pultr, \emph{Nearness, subfitness and sequential regularity,} Appl. Categ. Structures \textbf{8} (2000), 67--80.

\bibitem{ATOM}
J.R.~Isbell,~\emph{Atomless Parts of Spaces}, Math. Scand. \textbf{31} (1972), 5--32.

\bibitem{Descp}
J.R.~Isbell, \emph{First steps in descriptive theory of locales}, Trans. Amer. Math. Soc. \textbf{327} (1991), 353--371.

\bibitem{DOV}
P.~T.~Johnstone, \emph{Topos theory}, London Mathematical Society Monographs, vol. 10, Academic Press, London-New York, 1977.

\bibitem{PJ} P.~T.~Johnstone, \emph{Conditions related to De Morgan's law}, in: Applications of Sheaves, Lecture Notes in Math. \textbf{753}, Springer, Berlin (1980), pp. 479--491.

\bibitem{STONE}
P.~T.~Johnstone, \emph{Stone Spaces}, Cambridge Studies in Advanced Mathematics, vol. 3, Cambridge University Press, Cambridge, 1982.

\bibitem{SKET}
P.~T.~Johnstone, \emph{Sketches of an Elephant:  a Topos Theory Compendium},  vols. 1 and 2,  Oxford Logic Guides, vol. 44, The Clarendon Press, Oxford University Press, 2002.

\bibitem{SHEAVES}
S.~Mac Lane and I.~Moerdijk, \emph{Sheaves in Geometric and Logic. A First Introduction to Topos theory}, corrected reprint of the 1992 edition, Universitext, Springer-Verlag, New York, 1994.

\bibitem{PP12}
J.~Picado and A.~Pultr, \emph{Frames and locales: Topology without points}, Frontiers in Mathematics, vol. 28, Springer, Basel, 2012.

\bibitem{SUB} J.~Picado and A.~Pultr, \emph{New aspects of subfitness in frames and spaces}, Appl. Categ. Structures \textbf{24} (2016), 703--714.

 \bibitem{PT}
J.~Picado,  A.~Pultr and A.~Tozzi, \emph{Joins of closed sublocales}, Houston J. Math. \textbf{45} (2019), 21--38.

\bibitem{HORD}
T.~Plewe, \emph{Higher order dissolutions and Boolean coreflections of locales}, J. Pure Appl. Algebra \textbf{154} (2000), 273--293.

\bibitem{SLAT}
T.~Plewe, \emph{Sublocale lattices}, J. Pure Appl. Algebra, \textbf{168} (2002), 309–326.

\bibitem{CCD} R.~Rosebrugh and R.~J.~Wood, \emph{Constructive complete distributivity II}, Math. Proc. Cambridge Philos. Soc. \textbf{110} (1991), 245--249.

\bibitem{RS}J.~Rosick\'y and B.~\v{S}marda, \emph{$T_1$-locales}, Math. Proc. Cambridge Philos. Soc. \textbf{98} (1985), 81--86.
\end{thebibliography}
\end{document}